\documentclass[11pt,letterpaper,reqno]{amsart}
\usepackage[letterpaper,top=3cm, bottom=3cm, left=3.5cm, right=3.5cm]{geometry}

\usepackage[american]{babel} 
\usepackage[utf8]{inputenc} 

\usepackage{xcolor}
\usepackage[shortlabels]{enumitem}
\usepackage{hyperref}\hypersetup{colorlinks,linkcolor=.,citecolor=.,urlcolor=.}
 
\usepackage{booktabs}

\usepackage[T1]{fontenc}
\usepackage[sc]{mathpazo}
\usepackage{tgpagella}

\usepackage{amsmath,amssymb,amsthm,mathtools}

\usepackage{xargs}
\usepackage[colorinlistoftodos,prependcaption,textsize=tiny]{todonotes}
\newcommandx{\td}[2][1=]{\todo[linecolor=orange,backgroundcolor=orange!25,bordercolor=orange,#1]{#2}}

\newtheorem{theorem}{Theorem}[section]
\newtheorem{lemma}[theorem]{Lemma}
\newtheorem{proposition}[theorem]{Proposition}
\newtheorem{corollary}[theorem]{Corollary}

\newcommand{\RR}{\mathbb{R}}
\newcommand{\FF}{\mathbb{F}}

\newcommand{\cP}{\mathcal{P}}

\renewcommand{\epsilon}{\varepsilon}

\DeclareMathOperator{\sign}{sign}

\newcommand{\hermitianpol}[1]{\prescript{2}{}{A}_{{#1}}}
\newcommand{\hyperbolichalf}{\tfrac{1}{2}D_m}

\newcommand{\elliptic}[1]{\prescript{2}{}{D}_{{#1}}}

\DeclareMathOperator{\rk}{rank}

\newcommand{\qbin}[2]{\genfrac{[}{]}{0pt}{}{{#1}}{{#2}}_q}
\newcommand{\bbin}[2]{\genfrac{[}{]}{0pt}{}{{#1}}{{#2}}_b}
\newcommand{\pbin}[2]{\genfrac{[}{]}{0pt}{}{{#1}}{{#2}}_p}
\newcommand{\qqbin}[2]{\genfrac{[}{]}{0pt}{}{{#1}}{{#2}}_{q^2}}
\newcommand{\gaussian}[3]{\genfrac{[}{]}{0pt}{}{{#2}}{{#3}}_{{#1}}}

\newcommand{\poch}[3]{({#1};{#2})_{{#3}}}
\newcommand{\pochq}[2]{\poch{{#1}}{q}{{#2}}}
\newcommand{\pochb}[2]{({#1})_{{#2}}}

\newcommand{\hyp}[3]{\,\mbox{}_{3}\phi_{2}
	\left(\!\!\left.\begin{array}{c}#1\\#2\end{array}\right|#3\right)}
\newcommand{\hyper}[5]{\,\mbox{}_{{#1}}\phi_{{#2}}
	\left(\!\!\left.\begin{array}{c}#3\\#4\end{array}\right|#5\right)}

\newcommand{\qandq}{\quad\text{and}\quad}
\newcommand{\qforall}{\quad\text{for all }}

\title{Packings and Steiner systems in polar spaces}

\author{Kai-Uwe Schmidt and Charlene Wei\ss}
\address{Department of Mathematics, Paderborn University, Warburger Str.\ 100, 33098 Paderborn, Germany.}
\email[K.-U. Schmidt]{kus@math.upb.de}
\email[C. Wei\ss]{chweiss@math.upb.de}

\date{13 March 2022 (revised 20 December 2022)} 


\begin{document}
\begin{abstract}
A finite classical polar space of rank $n$ consists of the totally isotropic subspaces of a finite vector space equipped with a nondegenerate form such that $n$ is the maximal dimension of such a subspace. A $t$-Steiner system in a finite classical polar space of rank $n$ is a collection $Y$ of totally isotropic $n$-spaces such that each totally isotropic $t$-space is contained in exactly one member of $Y$. Nontrivial examples are known only for $t=1$ and $t=n-1$. We give an almost complete classification of such $t$-Steiner systems, showing that such objects can only exist in some corner cases. This classification result arises from a more general result on packings in polar spaces.
\end{abstract}
\maketitle

\section{Introduction}

A \emph{$t$-Steiner system} is a collection $Y$ of $n$-subsets of a $v$-set $V$ such that each $t$-subset of~$V$ is contained in exactly one member of~$Y$. The long-standing existence question for $t$-Steiner systems has been settled recently: it was shown in~\cite{Kee2014} and~\cite{GloKuhLoOst2016} that, for all $t\le n$ and all sufficiently large $v$, a $t$-Steiner system exists, provided that some natural divisibility conditions are satisfied.
\par
It is well known that combinatorics of sets can be regarded as the limiting case $q\to 1$ of combinatorics of vector spaces over the finite field $\FF_q$. Indeed, following~\cite{Cam1974} and~\cite{Del1978}, a \emph{$t$-Steiner system over $\FF_q$} is a collection~$Y$ of $n$-dimensional subspaces ($n$-spaces for short) of a $v$-space~$V$ over $\FF_q$ such that each $t$-space of~$V$ is contained in exactly one member of~$Y$. It is remarkable that, in the nontrivial case $1<t<n<v$, Steiner systems over $\FF_q$ are only known for a single set of parameters~\cite{BraEtzOstVarWas}, namely for $(t,n,v)=(2,3,13)$ and $q=2$.
\par
We may consider these objects as $q$-analogs of Steiner systems of type~$A_{n-1}$, as $V$ together with the action of $GL(n,q)$ is of this type. We study $q$-analogs of Steiner systems in finite vector spaces of types $\hermitianpol{2n-1}$, $\hermitianpol{2n}$, $B_n$, $C_n$, $D_n$, and $\elliptic{n+1}$ (using the notation of~\cite{Carter}). In each case, the space $V$ (defined over $\FF_{q^2}$ for $\hermitianpol{2n-1}$ and $\hermitianpol{2n}$) is equipped with a nondegenerate form and the relevant groups are $U(2n,q^2)$, $U(2n+1,q^2)$, $O(2n+1,q)$, $Sp(2n,q)$, $O^+(2n,q)$, and $O^-(2n+2,q)$, respectively. The notation is chosen such that the maximal totally isotropic subspaces of $V$, called \emph{generators}, have dimension $n$ (see Table~\ref{table:polarspaces}). A collection of all totally isotropic subspaces with respect to a given form is a \emph{finite classical polar space} (or just \emph{polar space}) of \emph{rank} $n$ and we denote this space by the same symbol as the type of the underlying vector space. A \emph{$t$-Steiner system} (of $n$-spaces) in a polar space~$\cP$ of rank $n$ is a collection~$Y$ of generators in~$\cP$ such that each totally isotropic $t$-space of~$V$ is contained in exactly one member of~$Y$. These objects are sometimes called \emph{regular systems} or \emph{$1$-regular systems} in the literature.
\par
A $1$-Steiner system in a polar space is known as a \emph{spread}, whose existence question has been studied for decades (see~\cite{Seg1965},~\cite{Dye1977},~\cite{Tha1981},~\cite{Kan1982a},~\cite{Kan1982b}, \cite{CalCamKanSei1997}, for example), but is still not fully resolved (see~\cite[\S~7.4]{HirschfeldThas} for the current status). The only other known nontrivial $t$-Steiner systems in polar spaces occur for $t=n-1$ in $D_n$ and equal one of the two bipartite halves of $D_n$ (see Section~\ref{sec:association schemes}), which are the two orbits under $SO^+(2n,q)$ acting on the generators of $D_n$~\cite[Thm.~11.61]{Taylor}.
\par
We prove the following classification result. 
\begin{theorem}\label{thm:Steinersystems_dualpolarspaces}
	Suppose that a polar space $\cP$ of rank $n$ contains a $t$-Steiner system with $1<t<n$. Then one of the following holds
	\begin{enumerate}[(1)]
		\item $t=2$ and $\cP=\hermitianpol{2n}$ or $\elliptic{n+1}$ for odd $n$,
		\item $t=n-1$ and $\cP=\hermitianpol{2n}$ or $\elliptic{n+1}$ for $q\geq 3$, or $\cP=D_n$.
	\end{enumerate} 
\end{theorem}
\par
Regarding the possible cases that remain in Theorem~\ref{thm:Steinersystems_dualpolarspaces}, we conjecture that, if  a polar space~$\cP$ of rank $n$ contains a $t$-Steiner system with $1<t<n$, then $t=n-1$ and $\cP=D_n$. The special cases $(n,t)=(4,2)$ and $(n,t)=(5,3)$ in Theorem~\ref{thm:Steinersystems_dualpolarspaces} were recently proved in~\cite{CosMarPavSma2022} and the results in the cases $t=n-1$ are essentially known (see Case (C1) in Section~\ref{sec:Steinerdualpolar}). All other cases appear to be new. In fact we prove a result on packings that is stronger than Theorem~\ref{thm:Steinersystems_dualpolarspaces} in most cases. 
\par
An elementary counting argument shows that the size of a $t$-Steiner system in a polar space necessarily equals the total number of totally isotropic $t$-spaces divided by the number of \mbox{$t$-spaces} contained in a generator. Our proof of Theorem~\ref{thm:Steinersystems_dualpolarspaces} is based on the fact that a set $Y$ of generators in a polar space is a $t$-Steiner system if and only if $Y$ has the correct size and $\dim U\cap W<t$ for all distinct $U,W\in Y$. Accordingly we define a \emph{$d$-code} in a polar space $\cP$ to be a set of generators~$Y$ of $\cP$, such that $n-\dim U\cap W\ge d$ for all distinct $U,W\in Y$ (here $(U,W)\mapsto n-\dim U\cap W$ agrees with the subspace metric used by coding theorists). Our main result, Theorem~\ref{thm:bound_herm_hyp} and Corollary~\ref{cor:bounds_dualpolarspaces}, is a bound on the size of a $d$-code in a polar space, which is sharp up to a constant factor in many cases. The bound is obtained using the concept of an association scheme and the powerful method of linear programming. It is then not hard to show that in most cases the bound is too small for a $t$-Steiner system to exist, eventually leading to Theorem~\ref{thm:Steinersystems_dualpolarspaces}. Numerical evidence suggests that in all cases remaining in Theorem~\ref{thm:Steinersystems_dualpolarspaces} the linear programming bound in the corresponding association scheme equals the size of the corresponding putative Steiner system. Hence it seems that entirely new techniques are required to deal with the remaining cases.
\par
We organize this paper as follows. In Section~\ref{sec:association schemes} we recall relevant facts about association schemes in general and about association schemes arising from polar spaces in particular. In Section~\ref{sec:codesdualpolar} we obtain bounds on the size of $d$-codes and in Section~\ref{sec:Steinerdualpolar} we show that, in most cases, these bounds are smaller than the size of a corresponding Steiner system. Some corner cases will be treated separately.


\section{The association schemes of polar spaces}\label{sec:association schemes}

We start this section with some basic facts about association schemes. For further information, we refer to~\cite{DelsartePhD},~\cite{BanIto1984},~\cite{BrouwerCohenNeumaier}, and~\cite{BannaiIto2021}. We will then introduce polar spaces and their associated association schemes.
\par
A (symmetric) \emph{association scheme $(X,(R_i))$ with $n$ classes} is a finite set $X$ together with $n+1$ nonempty relations $R_0,R_1,\dots,R_n$ such that
\begin{enumerate}
	\item[(A1)] $R_0$ is the identity relation and all $n+1$ relations partition $X\times X$,
	\item[(A2)] each relation is symmetric, that is, if $(x,y)\in R_i$, then $(y,x)\in R_i$,
	\item[(A3)] for every pair $(x,y)\in R_k$, the number of $z\in X$ with $(x,z)\in R_i$ and $(z,y)\in R_j$ is a constant $p_{ij}^k$ depending only on $i$, $j$, and $k$, but not on the particular choice of $(x,y)$.
\end{enumerate}
Let $(X,(R_i))$ be an association scheme with $n$ classes. For each relation $R_i$, the adjacency matrix of the graph $(X,R_i)$ is denoted by $D_i$ (with $D_0$ being the identity matrix). The complex vector space spanned by the matrices $D_0,D_1,\dots,D_n$ is a commutative matrix algebra of dimension $n+1$, called the \emph{Bose-Mesner algebra} of the association scheme. There exists a unique basis of minimal idempotent matrices $E_0(=1/|X|J),E_1,\dots,E_n$ for this algebra, where $J$ denotes the all one matrix. A change of bases yields the existence of unique real numbers $P_i(k)$ and $Q_k(i)$ such that
\[
D_i=\sum_{k=0}^n P_i(k) E_k \qandq E_k=\frac{1}{|X|}\sum_{i=0}^n Q_k(i)D_i.
\]
The numbers $P_i(k)$ and $Q_k(i)$ are called \emph{$P$-numbers} and \emph{$Q$-numbers} of the association scheme $(X,(R_i))$, respectively. Write $v_i=P_i(0)$ and $\mu_k=Q_k(0)$, which are called \emph{valencies} and \emph{multiplicities} of the association scheme, respectively. Indeed $P_i(k)$ is an eigenvalue of the graph $(X,R_i)$, each column of $E_k$ is a corresponding eigenvector, and the rank of $E_k$ equals $\mu_k$. In particular $v_i$ is the valency of the graph $(X,R_i)$. The $P$- and $Q$-numbers satisfy
\begin{gather}
	\mu_k P_i(k)=v_i Q_k(i)\quad\text{for all $i,k=0,1,\dots,n$}\label{eq:AskeyWilsonDuality}\\
	\frac{1}{|X|}\sum_{k=0}^n P_i(k)Q_k(j)=\delta_{ij}\qforall i,j=0,1,\dots,n,\label{eq:PQnumorth}
\end{gather}
where $\delta_{ij}$ denotes the Kronecker $\delta$-function.
\par
An association scheme is called \emph{$P$-polynomial} with respect to the ordering $R_0,R_1,\dots,R_n$ if there exist polynomials $f_i\in\RR[x]$ of degree $i$ and distinct real numbers $y_0, y_1,\dots,y_n$ such that $P_i(k)=f_i(y_k)$ for all $i,k=0,1,\dots,n$. Similarly an association scheme is called \emph{$Q$-polynomial} with respect to the ordering $E_0,E_1,\dots,E_n$ if there exist polynomials $g_k\in\RR[x]$ of degree $k$ and distinct real numbers $z_0,z_1,\dots,z_n$ such that $Q_k(i)=g_k(z_i)$ for all $i,k=0,1,\dots,n$.
\par
Next we will introduce polar spaces. Let~$V$ be a vector space over a finite field with~$p$ elements  equipped with a nondegenerate form~$f$. A subspace $U$ of $V$ is called \emph{totally isotropic} if \mbox{$f(u,w)=0$} for all $u,w\in U$, or in the case of a quadratic form, if $f(u)=0$ for all $u\in U$. A \emph{finite classical polar space} with respect to a form~$f$ consists of all totally isotropic subspaces of~$V$. It is well known that all maximal (with respect to the dimension) totally isotropic spaces in a polar space have the same dimension, which is called the \emph{rank} of the polar space. The maximal totally isotropic spaces are called \emph{generators}. A finite classical polar space~$\cP$ has the \emph{parameter}~$e$ if every $(n-1)$-space in~$\cP$ lies in exactly $p^{e+1}+1$ generators. Up to isomorphism, there are exactly six finite classical polar spaces of rank~$n$, which are listed together with their parameter~$e$ in Table~\ref{table:polarspaces}. We refer to~\cite{Cam1992},~\cite{Taylor},~\cite[\S~9.4]{BrouwerCohenNeumaier},~\cite[\S~4.2]{Bal2015}, and~\cite[\S~5.1]{HirschfeldThas} for further background on polar spaces. (We emphasize that, in this paper, the term dimension is used in the usual sense as vector space dimension, not as projective dimension sometimes used by geometers.)
\begin{table}[ht]
	\caption{List of all six finite classical polar spaces.}
	\centering
	\renewcommand*{\arraystretch}{1.25}
	\begin{tabular}{cccccccc}
		\toprule[0.4mm]
		\textbf{name} & \textbf{form} & \textbf{type} & \textbf{group} & $\dim(V)$ & $p$ & $e$\\
		\midrule[0.4mm]
		Hermitian & Hermitian & $\hermitianpol{2n-1}$ & $U(2n,q^2)$ & $2n$ & $q^2$ & $-1/2$\\ 
		Hermitian & Hermitian & $\prescript{2}{}{A}_{2n}$ & $U(2n+1,q^2)$ & $2n+1$ & $q^2$ & $1/2$\\
		symplectic & alternating & $C_n$ & $Sp(2n,q)$ & $2n$ & $q$ &  $0$\\
		hyperbolic & quadratic & $D_n$& $O^+(2n,q)$ & $2n$ & $q$ & $-1$\\
		parabolic & quadratic & $B_n$ & $O(2n+1,q)$ & $2n+1$ & $q$ & $0$\\
		elliptic & quadratic & $\prescript{2}{}{D}_{n+1}$ & $O^-(2n+2,q)$ & $2n+2$ & $q$ & $1$\\
		\bottomrule[0.4mm]
	\end{tabular}
	\label{table:polarspaces}
\end{table}
\par
Let $X$ consist of all generators of a polar space of rank~$n$ and define the relations
\begin{equation}\label{eqn:def_Ri}
	R_i=\{(U,W)\in X\times X:\dim(U\cap W)=n-i\}\quad\text{for $i=0,1,\dots,n$}.
\end{equation}
Then $(X,(R_i))$ is an association scheme with $n$ classes (see~\cite[\S~3.6]{BanIto1984},~\cite[\S~9.4]{BrouwerCohenNeumaier}, \cite[\S~6.4]{BannaiIto2021}, for example). It is well known that
\begin{equation}\label{eqn:size_X}
	|X|=\prod_{i=1}^n (1+p^{i+e}).
\end{equation}
Defining the \emph{$q$-binomial coefficient}
\[
\qbin{n}{k}=\prod_{j=1}^k \frac{q^{n-j+1}-1}{q^j-1}
\]
for nonnegative integers $n,k$, the $P$-numbers of $(X,(R_i))$ are given by~\cite[(8.1)]{StantonSome}, \cite[Prop.~2.4]{StantonThree}
\begin{equation}\label{eq:Pnum_dualpolarspaces}
	P_i(k)=v_i\pbin{n}{k}^{-1}\sum_{\ell=0}^i (-1)^\ell \pbin{n-i}{k-\ell}\pbin{i}{\ell} p^{\ell(\ell-i-e-1)},
\end{equation}
where
\begin{equation}\label{eqn:valuencies}
	v_i=p^{\binom{i+1}{2}+ie}\pbin{n}{i}
\end{equation}
are the valencies.\footnote{It should be noted that $p$ is assumed to be odd in~\cite{StantonSome} and~\cite{StantonThree}. However all parameters of the association scheme, as well as~$P_i(k)$ are polynomials in $p$, and hence the expression for $P_i(k)$ holds for all~$p$.} Note that~\eqref{eqn:def_Ri} and~\eqref{eq:Pnum_dualpolarspaces} impose an ordering on $E_0,E_1,\dots,E_n$, which we refer to as the \emph{standard ordering}.
\par
The $P$-number $P_i(k)$, given in~\eqref{eq:Pnum_dualpolarspaces}, is a polynomial of degree $i$ in $p^{-k}$, known as a $q$-Krawtchouk polynomial. The association scheme $(X,(R_i))$ is thus $P$-polynomial with respect to the ordering $R_0,R_1,\dots,R_n$. Moreover it is well known that $(X,(R_i))$ is also $Q$-polynomial with respect to the ordering $E_0,E_1,\dots,E_n$.
\par
We shall need the $P$-numbers also in a different form. We define the \emph{$q$-Pochhammer symbol}
\[
\pochq{a}{0}=1,\quad \pochq{a}{n}=\prod_{i=0}^{n-1}(1-aq^i)
\]
for a positive integer $n$ and a real number $a$ and the \emph{$q$-hypergeometric function} $\mbox{}_{r}\phi_{s}$ by
\[
\hyper{r}{s}{a_1,\dots,a_r}{b_1,\dots,b_s}{q,z}=\sum_{\ell=0}^{\infty} \frac{(a_1;q)_\ell\cdots(a_r;q)_\ell}{(b_1;q)_\ell\cdots(b_s;q)_\ell}\,(-1)^{(1+s-r)\ell} q^{(1+s-r)\binom{\ell}{2}}\,\frac{z^\ell}{(q;q)_\ell}.
\]
The $P$-numbers can now be written as~\cite[(8.1)]{StantonSome}
\begin{align}\label{eq:Pnum_dualpolarspaces_phi}
	P_i(k)=v_i\hyp{p^{-k},p^{-i},-p^{-n-e-1+k}}{0,p^{-n}}{p;p}.
\end{align}
\par
We close this section by noting that~$D_n$ gives rise to another association scheme, called the \emph{bipartite half of $D_n$}, in the following way. Let $X$ be the set of generators of $D_n$ and define two generators in $X$ to be equivalent if the dimension of their intersection has the same parity as $n$. This induces two equivalence classes, $X_1$ and $X_2$, and each pair $(X_i,(R_{2j})_{0\le j\le \lfloor\frac n2\rfloor})$ is a $P$- and $Q$-polynomial association scheme~\cite[\S~9.4.C]{BrouwerCohenNeumaier}, denoted by~$\frac 12 D_n$. Since $e=-1$ for~$D_n$, this also shows that $X_1$ and $X_2$ are $(n-1)$-Steiner systems in $D_n$.


\section{Codes in polar spaces}\label{sec:codesdualpolar}

Let $(X,(R_i))$ be a $P$- and $Q$-polynomial scheme with $n$ classes. We say that a subset~$Y$ of $X$ is a \emph{$d$-code} if no pair $(x,y)\in Y\times Y$ lies in one of the relations $R_1,\dots,R_{d-1}$. In particular, if $X$ is the set of generators in a polar space of rank $n$ and $R_i$ is given in~\eqref{eqn:def_Ri}, then a subset $Y$ of $X$ is a $d$-code if $n-\dim U\cap W\geq d$ for all distinct $U,W$ in~$Y$.
\par
The main goal of this section is to obtain upper bounds on the size of $d$-codes in the latter cases. To do so, we associate with each subset~$Y$ of $X$ two tuples of rational numbers. The \emph{inner distribution} of~$Y$ is the tuple $(A_0,A_1,\dots,A_n)$, where
\[
A_i=\frac{|(Y\times Y)\cap R_i|}{|Y|}.
\]
We then have $A_0=1$ and $\sum_{i=0}^nA_i=|Y|$. Note that $A_1=\cdots=A_{d-1}=0$ if and only if~$Y$ is a $d$-code. The \emph{dual distribution} of $Y$ is the tuple $(A'_0,A'_1,\dots,A'_n)$, where
\begin{align}\label{eq:dualdistr}
	A'_k=\sum_{i=0}^n Q_k(i)A_i.
\end{align}
Since $Q_0(i)=1$, we obtain $A'_0=|Y|$. Moreover we have
\begin{equation}\label{eq:dualDistrnonnegative}
	A'_k\ge 0 \quad\text{for all }k=0,1,\dots,n
\end{equation}
(see~\cite[Thm.~3.3]{DelsartePhD}, for example). 
\par
To derive bounds for $d$-codes in polar spaces, we begin with bounds for $d$-codes in~$\hermitianpol{2n-1}$ and the bipartite half $\frac12 D_n$ of~$D_n$ in Theorem~\ref{thm:bound_herm_hyp}. We proceed in this way because by taking a different $Q$-polynomial ordering for~$\hermitianpol{2n-1}$ and studying $\frac12 D_n$ instead of~$D_n$, we can express the resulting $Q$-numbers by $q$-Hahn polynomials. This allows us to treat~$\hermitianpol{2n-1}$ and $\frac12 D_n$ in a unified way. We will then use the bounds in~$\hermitianpol{2n-1}$ and~$\frac12 D_n$ to establish bounds for codes in the remaining polar spaces. We write
\begin{align}\label{eq:bc_hermpol_halfhyp}
	(b,c)=\begin{cases}
		(-q,-1) &\text{for $\hermitianpol{2n-1}$}\\
		(q^2,1/q) &\text{for $\hyperbolichalf$ if $m$ is even}\\
		(q^2,q) &\text{for $\hyperbolichalf$ if $m$ is odd},
	\end{cases}
\end{align}
and $(x)_i=(x;b)_i$ in what follows.
\par
\begin{theorem}\label{thm:bound_herm_hyp}
	Let $X$ be the set of generators in $\hermitianpol{2n-1}$ or $\hyperbolichalf$, where $n=\lfloor m/2\rfloor$ in the case of~$\hyperbolichalf$, and let $Y$ be a $d$-code in $X$ with $1\leq d\leq n$. Then we have
	\[
	|Y|\le \frac{|X|\pochb{q}{d-1}}{\pochb{qcb^n}{d-1}},
	\]
	where $d$ is required to be odd in the case of~$\hermitianpol{2n-1}$. For even $d$ in~$\hermitianpol{2n-1}$, we have
	\[
	|Y|\leq \frac{|X|\pochb{q}{d-1}}{\pochb{qcb^n}{d-1}}\;\frac{(b^{n-d+2}-1)+q\frac{b^{n+d-2}-1}{qb^{d-2}-1}(b^{n-d+1}-1)}{(b^{n-d+2}-1)+q\frac{b^{n+d-2}-1}{b^{n+d-1}-1}(b^{n-d+1}-1)}.
	\]
	Moreover these bounds also hold for $d$-codes in association schemes with the same $P$- and $Q$-numbers as $\hermitianpol{2n-1}$ or $\hyperbolichalf$.
\end{theorem}
\par
To prove Theorem~\ref{thm:bound_herm_hyp}, we first write the $Q$-numbers of $\hermitianpol{2n-1}$ and $\hyperbolichalf$ as $q$-Hahn polynomials. There exist different definitions for these polynomials. We will use the definition given in~\cite[\S~14.6]{Koekoek}. The \emph{$q$-Hahn polynomial} of degree~$k$ in the variable $q^{-x}$ with the parameters $n,A,B,C$ is given by
\[
\hyp{q^{-x},q^{-k},A q^k}{q^{-n},C^{-1}B^{-n}}{q;q}=\sum_{\ell\geq 0}\frac{\pochq{q^{-x}}{\ell}\pochq{q^{-k}}{\ell}\pochq{A q^k}{\ell}}{\pochq{q^{-n}}{\ell}\pochq{C^{-1}B^{-n}}{\ell}\pochq{q}{\ell}}\, q^\ell
\]
for $k=0,1,\dots,n$.
\par
The association scheme~$\hermitianpol{2n-1}$ is $Q$-polynomial with respect to two different orderings: the standard ordering $E_0,E_1,\dots,E_n$ and $E_0,E_n,E_1,E_{n-1},E_2,E_{n-2},\dots$~\cite{ChiharaStanton}. We continue to use $P_i(k)$ and $Q_k(i)$ to denote the $P$- and $Q$-numbers with respect to the standard ordering and we use $P'_i(k)$ and $Q'_k(i)$ to denote the $P$- and $Q$-numbers with respect to the second ordering. Then $P_i(k)$ is given in~\eqref{eq:Pnum_dualpolarspaces} and  $P'_i(k)$ is given by~\cite{ChiharaStanton}
\begin{alignat*}{4}
	P'_i(2k)&=P_i(k)&\quad&\text{for $i=0,\dots,n$ and $k=0,1,\dots,\Big\lfloor\frac{n}{2}\Big\rfloor$},\\
	P'_i(2k+1)&=P_i(n-k)&\quad&\text{for $i=0,\dots,n$ and $k=0,1,\dots,\Big\lfloor\frac{n-1}{2}\Big\rfloor$}.
\end{alignat*}
By applying the quadratic transformation for hypergeometric functions (see~\cite[(1.13.28)]{Koekoek}, for example)
\begin{align}\label{eq:quadTrafo_hypergeometric}
	\hyper{4}{3}{A^2, B^2, C, D}{ABq^{1/2}, -ABq^{1/2}, -CD}{q;q}
	=\hyper{4}{3}{A^2,B^2,C^2,D^2}{A^2B^2q,-CD,-CDq}{q^2;q^2}
\end{align}
to $P_i(k)$ and $P_i(n-k)$ given in~\eqref{eq:Pnum_dualpolarspaces_phi}, we obtain
\[
P'_i(k)=v_i\hyp{(-q)^{-i},(-q)^{-k},(-q)^{-2n+k-1}}{(-q)^{-n},-(-q)^{-n}}{{-q;-q}}.
\]
\par
Next we treat $\hyperbolichalf$. In this case we still denote by $P_i(k)$ and $Q_k(i)$ the $P$- and $Q$-numbers of~$D_m$ and by $P'_i(k)$ and $Q'_k(i)$ the $P$- and $Q$-numbers of~$\hyperbolichalf$. As in Theorem~\ref{thm:bound_herm_hyp}, put $n=\lfloor m/2\rfloor$. From~\cite{ChiharaStanton} we find that
\[
P'_i(k)=P_{2i}(k)\quad\text{for $i,k=0,1,\dots,n$}.
\]
Applying~\eqref{eq:quadTrafo_hypergeometric} to $P_{2i}(k)$ given in~\eqref{eq:Pnum_dualpolarspaces_phi} yields
\[
P'_i(k)=v_{2i}\hyp{q^{-2i},q^{-2k},q^{-2m+2k}}{q^{-m},q^{-m+1}}{q^2;q^2}.
\]
In summary, the $P$- and $Q$-numbers of $\hermitianpol{2n-1}$ and $\hyperbolichalf$ are given by
\begin{align}\label{eq:Pnum_Herm_hyp}
	P'_i(k)=v'_i  \hyp{b^{-i},b^{-k},q^{-1}c^{-1}b^{-2n+k}}{b^{-n},c^{-1}b^{-n}}{b;b}
\end{align}
and
\begin{align}\label{eq:Qnum_Herm_hyp}
	Q'_k(i)=\mu'_k \hyp{b^{-i},b^{-k},q^{-1}c^{-1}b^{-2n+k}}{b^{-n},c^{-1}b^{-n}}{b;b},
\end{align}
where the parameters $b$ and $c$ are stated in~\eqref{eq:bc_hermpol_halfhyp} and the remaining values are given in Table~\ref{table:Pnum_parameters}. Hence in both cases $Q'_k(i)$ is given by the $q$-Hahn polynomial of degree $k$ in $b^{-i}$.
\begin{table}[ht]
	\caption{Valencies and multiplicities occurring in~\eqref{eq:Pnum_Herm_hyp} and~\eqref{eq:Qnum_Herm_hyp}.}
	\centering
	\renewcommand*{\arraystretch}{1.25}
	\begin{tabular}{lccc} 
		\toprule[0.4mm]
		& $\hermitianpol{2n-1}$ & $\hyperbolichalf$ \\
		\midrule[0.4mm]
		$v'_i$ & $q^{i^2}\qqbin{n}{i}$ & $q^{\binom{2i}{2}}\qbin{m}{2i}$ \\
		$\mu'_k$ & \begin{tabular}{@{}c@{}}$\mu_{k/2}$ for $k$ even \\ $\mu_{n-(k-1)/2}$ for $k$ odd\end{tabular} & $\mu_k$ \\
		\bottomrule[0.4mm]
	\end{tabular}
	\label{table:Pnum_parameters}
\end{table}
\par
Before we prove Theorem~\ref{thm:bound_herm_hyp}, we record the following identity whose proof is deferred to the appendix. 
\begin{lemma}\label{lem:QnumIdentity_Herm_hyp}
	Let $X$ be the set of generators in $\hermitianpol{2n-1}$ or $\hyperbolichalf$, where we put $n=\lfloor m/2\rfloor$ in the latter case. Let $Q'_k(i)$ be given in~\eqref{eq:Qnum_Herm_hyp}. Then we have
	\[
	\sum_{k=0}^n b^{k(n-j)}\bbin{n-k}{n-j}\frac{\pochb{qcb^{n-k}}{n-j}}{\pochb{q}{n-j}} Q'_k(i)=|X| 	\bbin{n-i}{j}
	\]
	for all $i,j=0,1,\dots,n$.
\end{lemma}
\par
Now we prove Theorem~\ref{thm:bound_herm_hyp}.
\begin{proof}[Proof of Theorem~\ref{thm:bound_herm_hyp}]
	Suppose that $Y$ is a $d$-code in $\hermitianpol{2n-1}$ or $\hyperbolichalf$. Let $(A_0,A_1,\dots,A_n)$ and $(A'_0,A'_1,\dots,A'_n)$ be the inner and dual distribution of~$Y$, respectively, in terms of the orderings imposed by the $P$- and $Q$-numbers given in~\eqref{eq:Pnum_Herm_hyp} and~\eqref{eq:Qnum_Herm_hyp}. From~\eqref{eq:dualdistr} and Lemma~\ref{lem:QnumIdentity_Herm_hyp} we obtain for all $j=0,1,\dots,n$,
	\begin{align}
		\sum_{k=0}^j b^{k(n-j)}\bbin{n-k}{n-j}\frac{\pochb{qcb^{n-k}}{n-j}}{\pochb{q}{n-j}} A'_k
		&=\sum_{i=0}^n A_i\sum_{k=0}^j b^{k(n-j)}\bbin{n-k}{n-j}\frac{\pochb{qcb^{n-k}}{n-j}}{\pochb{q}{n-j}} Q'_k(i)\notag\\
		&=|X|\sum_{i=0}^n A_i\bbin{n-i}{j}.\label{eq:innerdualidentity_Herm_hyp}
	\end{align}
	First assume that $d$ is odd in the case of $\hermitianpol{2n-1}$. Since $A_1=\dots=A_{d-1}=0$ and $\bbin{n-i}{n-d+1}=0$ for $i\geq d$, we find from~\eqref{eq:innerdualidentity_Herm_hyp} with $j=n-d+1$ that
	\[
	\sum_{k=0}^{n-d+1}b^{k(d-1)}\bbin{n-k}{d-1}\frac{\pochb{qcb^{n-k}}{d-1}}{\pochb{q}{d-1}}A'_k=|X|\bbin{n}{d-1}A_0.
	\]
	Since $A_0=1$ and $A'_0=|Y|$, we obtain
	\begin{align}\label{eq:derive_bound_hermpol_halfhyp}
		\sum_{k=1}^{n-d+1} b^{k(d-1)}\bbin{n-k}{d-1}\frac{\pochb{qcb^{n-k}}{d-1}}{\pochb{q}{d-1}}A'_k
		=\bbin{n}{d-1}\left(|X|-\frac{\pochb{qcb^n}{d-1}}{\pochb{q}{d-1}}|Y|\right).
	\end{align}
	Recall that $A'_k\ge 0$. For $\hermitianpol{2n-1}$, the sign of $\pochb{qcb^{n-k}}{d-1}/\pochb{q}{d-1}$ is $(-1)^{(d-1)(n-k+1)}$ and the sign of $\bbin{n-k}{d-1}$ is $(-1)^{(d-1)(n-k-d+1)}$. Since $d$ is odd, both signs are thus positive. Hence all summands on the left-hand side of~\eqref{eq:derive_bound_hermpol_halfhyp} are nonnegative implying
	\[
	|Y|\leq \frac{|X|\pochb{q}{d-1}}{\pochb{qcb^n}{d-1}},
	\]
	as required.
	\par
	Now consider $\hermitianpol{2n-1}$ for even $d$. Put
	\begin{align*}
		x_k&=b^{k(d-1)+d-2}\frac{\pochb{b^{n-k+1}}{d-1}\pochb{b^n}{d-2}}{\pochb{q}{d-1}\pochb{q}{d-2}}\bbin{n-k}{d-1}\bbin{n-1}{d-2},\\
		y_k&=b^{k(d-2)+d-1}\frac{\pochb{b^{n-k+1}}{d-2}\pochb{b^n}{d-1}}{\pochb{q}{d-2}\pochb{q}{d-1}}\bbin{n-k}{d-2}\bbin{n-1}{d-1}.
	\end{align*}
	Use~\eqref{eq:innerdualidentity_Herm_hyp} with $j=n-d+1$ and $j=n-d+2$ to obtain
	\begin{multline}\label{eq:hermpoldeven_1}
		\sum_{k=0}^{n-d+2}(x_k-y_k)A'_k\\
		=|X|b^{d-2}\frac{\pochb{b^n}{d-2}}{\pochb{q}{d-2}}\left(\bbin{n-1}{d-2}\bbin{n}{d-1}+q\bbin{n-1}{d-1}\bbin{n}{d-2}\frac{b^{n+d-2}-1}{qb^{d-2}-1}\right).
	\end{multline}
	Next we show that the summands on the left-hand side are nonnegative. The sign of $\bbin{m}{\ell}$ is $(-1)^{\ell(m-\ell)}$ and the sign of $\pochb{b^m}{\ell}/\pochb{q}{\ell}$ is $(-1)^{m\ell}$. Hence we have $\sign(x_k)=(-1)^k$ and $\sign(y_k)=-1$, which implies that the left-hand side of~\eqref{eq:hermpoldeven_1} equals
	\[
	\sum_{k=0}^{n-d+2}((-1)^k|x_k|+|y_k|)A'_k.
	\]
	From
	\[
	\frac{x_k}{y_k}=b^{k-1} \frac{(b^{n-k-d+2}-1)(b^{n-k+d-1}-1)}{(b^{n+d-2}-1)(b^{n-d+1}-1)},
	\]
	we find that $|x_k|\leq |y_k|$ for all $k\geq 1$. Therefore the left-hand side of~\eqref{eq:hermpoldeven_1} can be bounded from below by $(x_0-y_0)A'_0$, which is also positive. 	Since $A'_0=|Y|$, we thus find from~\eqref{eq:hermpoldeven_1} that
	\[
	|Y|\leq \frac{|X|b^{d-2}\frac{\pochb{b^n}{d-2}}{\pochb{q}{d-2}}\left(\bbin{n-1}{d-2}\bbin{n}{d-1}+q\bbin{n-1}{d-1}\bbin{n}{d-2}\frac{b^{n+d-2}-1}{qb^{d-2}-1}\right)}{\left(b^{d-2}\frac{\pochb{b^{n+1}}{d-1}\pochb{b^n}{d-2}}{\pochb{q}{d-1}\pochb{q}{d-2}}\bbin{n}{d-1}\bbin{n-1}{d-2}-b^{d-1}\frac{\pochb{b^{n+1}}{d-2}\pochb{b^n}{d-1}}{\pochb{q}{d-2}\pochb{q}{d-1}}\bbin{n}{d-2}\bbin{n-1}{d-1}\right)}.
	\]
	We can now deduce the second inequality of the theorem after elementary manipulations. This completes the proof.
\end{proof}
\par
In what follows we use Theorem~\ref{thm:bound_herm_hyp} to obtain bounds for $d$-codes in the remaining polar spaces $\hermitianpol{2n}$, $B_n$, $C_n$, $D_n$, and $\elliptic{n+1}$. To do so, we write
\[
\alpha(n,d)=\bigg(\prod\limits_{i=1}^n (1+q^{2i-1})\bigg)\bigg(\prod\limits_{i=1}^{d-1} \dfrac{1+(-q)^i}{1-(-q)^{n+i}}\bigg)\,\epsilon(n,d),
\]
where $\epsilon(n,d)=1$ for odd $d$ and
\[
\epsilon(n,d)=\frac{((-q)^{n-d+2}-1)+q\frac{(-q)^{n+d-2}-1}{q(-q)^{d-2}-1}((-q)^{n-d+1}-1)}{((-q)^{n-d+2}-1)+q\frac{(-q)^{n+d-2}-1}{(-q)^{n+d-1}-1}((-q)^{n-d+1}-1)}
\]
for even $d$, and
\begin{align*}
	\beta(m,d)=\begin{cases}
		\bigg(\prod\limits_{i=1}^{m-1} (1+q^i)\bigg)\bigg(\prod\limits_{i=1}^{d-1}\dfrac{1-q^{2i-1}}{1-q^{m+2i-2}}\bigg)&\text{for even $m$}\\[10pt]
		\bigg(\prod\limits_{i=1}^{m-1} (1+q^i)\bigg)\bigg(\prod\limits_{i=1}^{d-1}\dfrac{1-q^{2i-1}}{1-q^{m+2i-1}}\bigg)&\text{for odd $m$}.
	\end{cases}
\end{align*}
Observe that, using~\eqref{eqn:size_X}, the bounds in Theorem~\ref{thm:bound_herm_hyp} for $\hermitianpol{2n-1}$ and $\hyperbolichalf$ equal $\alpha(n,d)$ and $\beta(m,d)$, respectively.
\par
We make the following observation about $d$-codes in $D_n$ if $d$ is even.
\begin{proposition}\label{prop:Dn_codes_deven}
	Every $d$-code in $D_n$ with even $d$ and $2\leq d\leq n$ induces a $\frac d2$-code in $\frac 12 D_n$ of the same size.
\end{proposition}
\begin{proof}
	Recall that the set of generators in $D_n$ is partitioned into two equivalence classes $X_1$ and~$X_2$, where two generators lie in the same class if and only if the dimension of their intersection has the same parity as $n$. Let $Y$ be a $d$-code in $D_n$ with even $d$ and $2\leq d\leq n$. For each $y\in Y$, choose an $(n-1)$-space contained in~$y$. Since $d>1$, every two such $(n-1)$-spaces are distinct and the dimension of their intersection is at most $n-d$. Since $e=-1$ for $D_n$, each of these $(n-1)$-spaces lies in exactly two generators---one from~$X_1$ and one from $X_2$. Let $\widehat{Y}$ be the set of all generators in $X_1$ corresponding to the chosen $(n-1)$-spaces. Then we have
	\[
	\dim(x\cap y)\leq n-d
	\]
	for all $x,y\in\widehat{Y}$ since $\dim(x\cap y)$ must have the same parity as $n$. Hence $\widehat{Y}\subseteq X_1$ is a $\frac d2$-code in $\frac 12 D_n$ with $|Y|=|\widehat{Y}|$, as required. 
\end{proof}
We can now derive bounds for codes in all polar spaces.
\begin{corollary}\label{cor:bounds_dualpolarspaces}
	Let $\cP$ be a polar space of rank~$n$ and let $Y$ be a $d$-code in~$\cP$ with $1\leq d\leq n$. Put $\delta=\lceil d/2\rceil$. 
	\begin{enumerate}[(a)]
		\item If $\cP=\hermitianpol{2n-1}$, then $|Y|\leq \alpha(n,d)$.
		\item If $\cP=\hermitianpol{2n}$, then $|Y|\leq \alpha(n+1,d)$.
		\item If $\cP=B_n$ or $C_n$, then $|Y|\leq \beta(n+1,\delta)$.
		\item If $\cP=D_n$ and $d$ is odd, then $|Y|\leq 2\beta(n,\delta)$.
		\item If $\cP=D_n$ and $d$ is even, then $|Y|\leq \beta(n,\delta)$.
		\item If $\cP=\elliptic{n+1}$, then $|Y|\leq \beta(n+2,\delta)$.
	\end{enumerate}
\end{corollary}
\begin{proof}
	The bound in (a) follows directly from Theorem~\ref{thm:bound_herm_hyp} by using~\eqref{eqn:size_X}.
	\par
	A $d$-code in $D_n$ induces $\delta$-codes in each of the two bipartite halves of $D_n$, so it is at most twice as large as a $\delta$-code in $\frac 12 D_n$. Theorem~\ref{thm:bound_herm_hyp} then gives~(d) and Proposition~\ref{prop:Dn_codes_deven} implies (e).
	\par
	In the cases of $B_n$ and $C_n$, one obtains a new association scheme with the classes
	\[
	R_0, R_1\cup R_2, R_3\cup R_4, \dots.
	\]
	This new association scheme has the same $P$- and $Q$-numbers as $\frac 12 D_{n+1}$~\cite{Ivanovetal}. Therefore the size of a $d$-code in $B_n$ or $C_n$ is at most the upper bound for a $\delta$-code in $\frac 12 D_{n+1}$ given in Theorem~\ref{thm:bound_herm_hyp}, which yields (c).
	\par
	To establish the remaining cases~(b) and~(f), note that $\elliptic{n+1}$ and $\hermitianpol{2n}$ arise by intersecting $B_{n+1}$ and $\hermitianpol{2n+1}$, respectively, with a hyperplane. Hence $\elliptic{n+1}$ can be embedded into $B_{n+1}$ and $\hermitianpol{2n}$ can be embedded into $\hermitianpol{2n+1}$. Note that $B_{n+1}$ and $\hermitianpol{2n+1}$ are of rank $n+1$ and each generator in~$\elliptic{n+1}$ or~$\hermitianpol{2n}$ becomes an~$n$-space in~$B_{n+1}$ or~$\hermitianpol{2n+1}$ under these embeddings. In $B_{n+1}$ and $\hermitianpol{2n+1}$, every $n$-space is contained in exactly~$p^{e+1}+1=q+1$ generators. For each embedded element of $Y$, we choose one of these $q+1$ generators giving a subset~$\widetilde{Y}$ of~$B_{n+1}$ or~$\hermitianpol{2n+1}$. Then~$\widetilde{Y}$ is also a $d$-code and (c) implies (f) and (a) implies~(b).
\end{proof}
\par
We also have the following more useful bounds on $\alpha(n,d)$ and $\beta(n,d)$.
\begin{lemma}\label{lem:bounds_alpha_beta}
	For $1\le d\le n$, we have
	\begin{align}\label{eq:bound_alpha}
		\alpha(n,d)<\begin{cases}
			\frac{14}{5}q^{n(n-d+1)} & \text{for odd $d$}\\
			\frac{14}{5}q^{n(n-d+2)} & \text{for even $d$},
		\end{cases}
	\end{align}
	and
	\begin{align}\label{eq:bound_beta}
		\beta(n,d)<\begin{cases}
			\frac52q^{(n-1)(n-2d+2)/2} & \text{for even $n$}\\
			\frac52q^{n(n-2d+1)/2} & \text{for odd $n$}.
		\end{cases}
	\end{align}
\end{lemma}
To prove Lemma~\ref{lem:bounds_alpha_beta} we use the identity
\begin{align}\label{eq:ineqfraction}
	\frac{x-1}{y-1}\le \frac{x}{y}\quad\text{for $y\ge x>1$}
\end{align}
and the following lemma.
\begin{lemma}
	Let $n\geq 1$ and $q\geq 2$ be integers. Then we have
	\begin{align}\label{eq:ineq_product}
		\prod_{i=1}^n \left(1+\frac{1}{q^i}\right)<\frac52,\quad \prod_{i=1}^n\left(1+\frac{1}{q^{2i}}\right)<\frac 75 \qandq \prod_{i=1}^n\left(1+\frac{1}{q^{2i-1}}\right)<2.
	\end{align}
\end{lemma}
\begin{proof}
	We use $1+x<\exp(x)$ to obtain
	\[
	\prod_{i=1}^n \left(1+\frac{1}{q^i}\right)<\left(1+\frac 1q\right)\exp\left(\frac{1}{q(q-1)}\right)\leq \left(1+\frac 1q\right)\exp\left(\frac 1q\right).
	\]
	Applying $(1+x)\exp(x)<\frac 52$ for all $x\in[0,\frac12]$ yields the first inequality. Using a similar approach gives us
	\[
	\prod_{i=1}^n\left(1+\frac{1}{q^{2i}}\right)<\exp\left(\frac{1}{q^2-1}\right)\leq \exp\left(\frac 13\right)<\frac 75,
	\]
	and
	\[
	\prod_{i=1}^n\left(1+\frac{1}{q^{2i-1}}\right)<\exp\left(\frac{q}{q^2-1}\right)\leq \exp\left(\frac 23\right)<2,
	\]
	as required.
\end{proof}
We can now prove Lemma~\ref{lem:bounds_alpha_beta}.
\begin{proof}[Proof of Lemma~\ref{lem:bounds_alpha_beta}]
	For $\beta(n,d)$ and even $n$, use~\eqref{eq:ineqfraction} and~\eqref{eq:ineq_product} to obtain 
	\begin{align*}
		\beta(n,d)&<\left(\prod_{i=1}^{n-1} q^i\left(1+\frac{1}{q^i}\right)\right) q^{(-n+1)(d-1)}\\
		&\le \frac 52 q^{(n-1)(n-2d+2)/2}.
	\end{align*}
	The bound for $\beta(n,d)$ and odd $n$ can be obtained similarly. For $\alpha(n,d)$, we write
	\begin{align}\label{eq:alpha_1}
		\alpha(n,d)=\bigg(\prod_{i=1}^n (1+q^{2i-1})\bigg)\bigg(\prod_{i=1}^{d-1} \frac{q^i+(-1)^i}{q^{n+i}-(-1)^{n+i}}\bigg)(-1)^{(n+1)(d-1)}\epsilon(n,d).
	\end{align}
	We have
	\begin{equation}\label{eqn:prod_herm}
		\prod_{i=1}^{d-1} \frac{q^i+(-1)^i}{q^{n+i}-(-1)^{n+i}}
		=\begin{cases}
			\prod\limits_{i=1}^{\frac{d-1}{2}} \frac{(q^{2i}+1)(q^{2i-1}-1)}{(q^{n+2i}-(-1)^n)(q^{n+2i-1}+(-1)^n)}&\text{for odd $d$}\\
			\frac{q^{d-1}-1}{q^{n+d-1}+(-1)^n}\prod\limits_{i=1}^{\frac{d-2}{2}} \frac{(q^{2i}+1)(q^{2i-1}-1)}{(q^{n+2i}-(-1)^n)(q^{n+2i-1}+(-1)^n)}&\text{for even $d$}.
		\end{cases}
	\end{equation}
	Using~\eqref{eq:ineqfraction} and~\eqref{eq:ineq_product}, we obtain for each $r\ge 1$,
	\begin{align*}
		\prod_{i=1}^r \frac{(q^{2i}+1)(q^{2i-1}-1)}{(q^{n+2i}-(-1)^n)(q^{n+2i-1}+(-1)^n)}
		&\leq \prod_{i=1}^r \frac{(q^{2i}+1)(q^{2i-1}-1)}{(q^{n+2i}+1)(q^{n+2i-1}-1)}\\
		&\leq \prod_{i=1}^r  q^{-2n}\left(1+\frac{1}{q^{2i}}\right)\\
		&< \frac 75 q^{-2nr}.
	\end{align*}
	Substitute into~\eqref{eqn:prod_herm} to give
	\begin{align}\label{eq:herm_prod2}
		\prod_{i=1}^{d-1} \frac{q^i+(-1)^i}{q^{n+i}-(-1)^{n+i}}
		<\begin{cases}
			\frac 75 q^{-n(d-1)}&\text{for odd $d$}\\
			\frac 75 q^{-n(d-2)}\frac{q^{d-1}-1}{q^{n+d-1}+(-1)^n}&\text{for even $d$}.
		\end{cases}
	\end{align}
	From~\eqref{eq:ineq_product} we have
	\begin{align}\label{eq:herm_prod1}
		\prod_{i=1}^n (1+q^{2i-1})=\prod_{i=1}^n q^{2i-1}\left(1+\frac{1}{q^{2i-1}}\right)<2q^{n^2}.
	\end{align}
	Substitute~\eqref{eq:herm_prod2} and~\eqref{eq:herm_prod1} into~\eqref{eq:alpha_1} to obtain
	\begin{align}\label{eq:Rbound_herm}
		\alpha(n,d)<\begin{cases}
			\frac{14}{5}q^{n(n-d+1)}&\text{for odd $d$}\\
			\frac{14}{5}q^{n(n-d+2)}\frac{q^{d-1}-1}{q^{n+d-1}+(-1)^n}(-1)^{(n+1)(d-1)}\epsilon(n,d)&\text{for even $d$}.
		\end{cases}
	\end{align}
	For even $d$, we have
	\begin{align}
		(-1)^{(n+1)(d-1)}\epsilon(n,d)&=\frac{q\frac{q^{n+d-2}-(-1)^n}{q^{d-1}-1}(q^{n-d+1}+(-1)^n)-(-1)^n(q^{n-d+2}-(-1)^n)}{(q^{n-d+2}-(-1)^n)+q\frac{q^{n+d-2}-(-1)^n}{q^{n+d-1}+(-1)^n}(q^{n-d+1}+(-1)^n)}   \nonumber\\[1ex]
		&=\frac{q\frac{(q^{n+d-2}-(-1)^n)(q^{n-d+1}+(-1)^n)}{(q^{d-1}-1)(q^{n-d+2}-(-1)^n)}-(-1)^n}{q\frac{(q^{n+d-2}-(-1)^n)(q^{n-d+1}+(-1)^n)}{(q^{n+d-1}+(-1)^n)(q^{n-d+2}-(-1)^n)}+1}   \nonumber\\[1ex]
		&<\frac{q^{n+d-1}+(-1)^n}{q^{d-1}-1},\label{eq:eps_bound}
	\end{align}
	using~\eqref{eq:ineqfraction}, so that~\eqref{eq:Rbound_herm} gives the required bound for $\alpha(n,d)$.
\end{proof}
\par
We close this section by discussing the sharpness of the bounds in Corollary~\ref{cor:bounds_dualpolarspaces}. For a vector space $V$, let $P_n(V)$ be the set of $n$-spaces of $V$. Define a mapping
\begin{gather*}
	v:\FF_p^{n\times n}\to P_n(\FF_p^{2n})\\
	A\mapsto\left\{\binom{x}{Ax}:x\in\FF_p^n\right\}.
\end{gather*}
It is well known~\cite[\S~9.5.E]{BrouwerCohenNeumaier} that, after an appropriate choice of the form, $v(A)$ is in $\hermitianpol{2n-1}$ if and only if $A$ is Hermitian, $v(A)$ is in $C_n$ if and only $A$ is symmetric and $v(A)$ is in $D_n$ if and only if~$A$ is alternating, namely skew-symmetric with zero main diagonal (as before, $p=q^2$ for $\hermitianpol{2n-1}$ and $p=q$ otherwise). The mapping $v$ satisfies
\[
n-\dim(v(A)\cap v(B))=\rk(A-B)
\]
for all $A,B\in\FF_q^{n\times n}$, so in particular $v$ is injective. Accordingly define a subset $Z$ of $\FF_q^{n\times n}$ to be a \emph{$d$-code} if $\rk(A-B)\ge d$ for all distinct $A,B\in Z$. Such objects were studied in~\cite{Sch2018},~\cite{Sch2020}, and~\cite{DelGoe1975}, for Hermitian, symmetric, and alternating matrices, respectively.
\par
In particular from~\cite{Sch2018} and the injection $v$ we find that, for odd $d$, there exists a $d$-code~$Y$ in~$\hermitianpol{2n-1}$ satisfying $|Y|=q^{n(n-d+1)}$. In view of Lemma~\ref{lem:bounds_alpha_beta}, this shows that the bound in Corollary~\ref{cor:bounds_dualpolarspaces}~(a) for odd $d$ is sharp up to a constant factor. Likewise from~\cite{Sch2020} we find that, for odd $d$, there exists a $d$-code~$Y$ in~$C_n$ satisfying
\[
|Y|=\begin{cases}
	q^{(n+1)(n-d+1)/2} & \text{for even $n$}\\
	q^{n(n-d+2)/2} & \text{for odd $n$},
\end{cases}
\]
showing that the bound in Corollary~\ref{cor:bounds_dualpolarspaces}~(c) for $\cP=C_n$ and odd $d$ is sharp up to a constant factor. Since $B_n$ and $C_n$ are isomorphic for even $q$ (see~\cite[\S~6.4]{BannaiIto2021}, for example), the same is true when $\cP=B_n$ and $q$ is even. From~\cite{DelGoe1975} we find that, for even $d$, there exists a $d$-code~$Y$ in~$D_n$ satisfying
\[
|Y|=\begin{cases}
	q^{(n-1)(n-d+2)/2} & \text{for even $n$ and even $q$}\\
	q^{n(n-d+1)/2} & \text{for odd $n$}.
\end{cases}
\]
Since a $d$-code is trivially also a $(d-1)$-code, this shows that the bound in Corollary~\ref{cor:bounds_dualpolarspaces}~(d) and~(e) is sharp up to a constant factor except possibly when $n$ is even and $q$ is odd. In all other cases one can obtain constructions of $d$-codes in a similar fashion, showing that the remaining bounds in Corollary~\ref{cor:bounds_dualpolarspaces} are met up to a small power of $q^n$.

\section{Nonexistence of Steiner systems in polar spaces}\label{sec:Steinerdualpolar}

We now prove Theorem~\ref{thm:Steinersystems_dualpolarspaces}. The proof is split into the following cases:
\begin{enumerate}[(C1)]
	\item $t=n-1$ and $\cP=\hermitianpol{2n}, \elliptic{n+1}$ for $q=2$ or $\cP=\hermitianpol{2n-1},B_n,C_n$,
	\item $\cP=D_n$ with $1<t<n-1$,
	\item $\cP=B_n$ or $C_n$ with $t=2$ and even $n$ or $2<t<n-1$,
	\item $\cP=\elliptic{n+1}$ with $t\in\{2,3\}$ and odd $n$ or $3<t<n-1$, but $(n,t)\not\in\{(7,4),(8,5)\}$,
	\item $\cP=\hermitianpol{2n-1}$ with $1<t<n-1$,
	\item $\cP=\hermitianpol{2n}$ with $t=2$ and even $n$, or $2<t<n-1$ except for  $(n,t)=(6,3)$,
	\item $t=2$ and $\cP=B_n$ or $C_n$ for odd $n>3$ or $\cP=\elliptic{n+1}$ for even $n>3$,
	\item $\cP=\elliptic{n+1}$ with $t=3$ and even $n>4$,
	\item $\cP=\elliptic{n+1}$ with $(n,t)=(7,4)$ or $(8,5)$, or $\cP=\hermitianpol{2n}$ with $(n,t)=(6,3)$.
\end{enumerate}
\par
The case (C1) is essentially known~\cite[p.~160]{VanhovePhD} and a proof is sketched below for completeness. The cases (C2)--(C6) will follow from Theorem~\ref{thm:bound_herm_hyp} and Corollary~\ref{cor:bounds_dualpolarspaces}. The cases (C7)--(C9) are some corner cases, which need special treatment.
\par
We begin with a sketch for a proof of (C1).
\begin{proof}[Proof of (C1)]
	By taking the elements of an $(n-1)$-Steiner system in a polar space of rank $n$ that contain a fixed isotropic $1$-space $v$ and taking the quotient by $v$, one obtains an $(n-2)$-Steiner system in a polar space of the same type but rank $n-1$. This reduces the existence question to $2$-Steiner systems in rank $3$ or $1$-Steiner systems, namely spreads, in rank $2$. There are no spreads in $B_2$ for odd~$q$, $\hermitianpol{4}$ for $q=2$, and $\hermitianpol{5}$ for all~$q$~\cite[\S~7.4]{HirschfeldThas} and there are no $2$-Steiner systems in $\elliptic{4}$ for $q=2$~\cite{Panigrahi} and $C_3$ for all~$q$~\cite{Thomas}, \cite{Cooperstein_etal}. Since $B_n$ and $C_n$ are isomorphic if $q$ is even (see~\cite[\S~6.4]{BannaiIto2021}, for example), there are also no $2$-Steiner systems in $B_3$ for even~$q$.
\end{proof}
\par
To prove (C2)--(C6), we recall that the number of totally isotropic $t$-spaces in a polar space of rank~$n$ is
\[
\pbin{n}{t}\prod_{i=0}^{t-1}(1+p^{n-i+e})
\]
(see~\cite[Lem.~9.4.1]{BrouwerCohenNeumaier}, for example). Since every generator contains exactly $\pbin{n}{t}$ subspaces of dimension~$t$, the size of a $t$-Steiner system is thus given by
\begin{equation}\label{eqn:size_Steiner_t}
	\prod_{i=0}^{t-1}(1+p^{n-i+e}).
\end{equation}
If $Y$ is a $t$-Steiner system, then the intersection of two distinct members of $Y$ can have dimension at most $t-1$, and so a $t$-Steiner system is an $(n-t+1)$-code. Henceforth we write $d=n-t+1$ and let $B$ denote the corresponding bound of a $d$-code in Corollary~\ref{cor:bounds_dualpolarspaces}. We denote the size of an $(n-d+1)$-Steiner system by~$S$, hence
\begin{equation}\label{eqn:size_Steiner_d}
	S=\prod_{i=0}^{n-d}(1+p^{n-i+e}),
\end{equation}
and in particular
\begin{equation}\label{eqn:bound_Steiner_d}
	S\ge p^{\frac12(n-d+1)(n+d+2e)}.
\end{equation}
We set $R=B/S$ and show that $R<1$.
\par
\begin{proof}[Proof of (C2)]
	In this case we assume that $\cP=D_n$ and $2<d<n$. Use Corollary~\ref{cor:bounds_dualpolarspaces}~(d) and (e), \eqref{eqn:bound_Steiner_d}, and~\eqref{eq:bound_beta} to obtain
	\begin{align}\label{eq:hyp_R}
		R<\begin{cases}
			\frac 52 q^{\frac 12 (d-2)(d-n)}&\text{for even $n$ and even $d$}\\
			5 q^{\frac 12(d-1)(d-n-1)}&\text{for even $n$ and odd $d$}\\
			\frac 52 q^{\frac 12 (d-2) (d-n-1)}&\text{for odd $n$ and even $d$}\\
			5 q^{\frac 12 (d-1)(d-n-2)}&\text{for odd $n$ and odd $d$.}
		\end{cases}
	\end{align}
	If $n$ and $d$ have the same parity, then~\eqref{eq:hyp_R} implies $R<1$. If $n$ and $d$ have a different parity, then~\eqref{eq:hyp_R} implies $R<1$, except when $(n,d)=(4,3)$. In the latter case, Corollary~\ref{cor:bounds_dualpolarspaces}~(d) and~\eqref{eqn:size_Steiner_d} give 
	\[
	R=\frac{2}{1+q^2}<1.
	\]
	This completes the proof.
\end{proof}
\begin{proof}[Proof of (C3)]
	In this case we assume that $\cP=B_n$ or $C_n$ and $2<d<n-1$ or $d=n-1$ is odd. Use Corollary~\ref{cor:bounds_dualpolarspaces}~(c),~\eqref{eqn:bound_Steiner_d}, and~\eqref{eq:bound_beta} to obtain
	\[
	R<\begin{cases}
		\frac 52 q^{\frac 12 \left(d(d-1)-(n+1)(d-2)\right)}&\text{for even $n$ and even $d$}\\
		\frac 52 q^{\frac 12 \left(d(d+1)-(n+1)(d-1)\right)}&\text{for even $n$ and odd $d$}\\
		\frac 52 q^{\frac 12 \left(d(d-1)-n(d-2)\right)}&\text{for odd $n$ and even $d$}\\
		\frac 52 q^{\frac 12 \left(d(d+1)-n(d-1)\right)}&\text{for odd $n$ and odd $d$.}
	\end{cases}
	\]
	It is the readily verified that $R<1$, except if (i) $d=4$ and $n=6,7$, or (ii) $d=3$ and $n=6,7$, or (iii) $d=n-2$ is odd,  or (iv) $d=n-1$ is odd. For (i) and (ii), Corollary~\ref{cor:bounds_dualpolarspaces}~(c) and~\eqref{eqn:size_Steiner_d} imply that $R$ equals $(1+q^3)/(1+q^4)$ and $1/(1+q^4)$, respectively, giving $R<1$ in both cases. For (iii), Corollary~\ref{cor:bounds_dualpolarspaces}~(c) and~\eqref{eqn:size_Steiner_d} imply that
	\begin{align*}
		R&=\bigg(\prod_{i=1}^{n-2} (1+q^i)\bigg)\bigg(\prod_{i=1}^{\frac n2-1}\frac{1-q^{2i-1}}{(1-q^{\frac n2+i})(1+q^{\frac n2+i})}\bigg)\\
		&=\frac{1}{1+q^{n-1}}\bigg(\prod_{i=1}^{\frac n2}(1+q^i)\bigg)\bigg(\prod_{i=1}^{\frac n2-1}\frac{1-q^{2i-1}}{1-q^{\frac n2+i}}\bigg)\\[1ex]
		&<\frac 52\, \frac{q}{1+q^{n-1}}<1,
	\end{align*}
	using~\eqref{eq:ineqfraction},~\eqref{eq:ineq_product}, and $n\ge 4$. Similarly, for (iv), we deduce
	\[
	R<\frac 52\, \frac{q}{1+q^{n-2}}<1,
	\]
	which completes the proof.
\end{proof}
\begin{proof}[Proof of (C4)]
	In this case we assume that $\cP=\elliptic{n+1}$ and $2<d<n-2$ or $d=n-2$ is odd or $d=n-1$ is even, but $(n,d)\not\in\{(7,4),(8,4)\}$. Use Corollary~\ref{cor:bounds_dualpolarspaces}~(e),~\eqref{eqn:bound_Steiner_d}, and~\eqref{eq:bound_beta} to obtain
	\[
	R<\begin{cases}
		\frac 52 q^{\frac 12 \left(d(d+1)-(n+1)(d-2)\right)}&\text{for even $n$ and even $d$}\\
		\frac 52 q^{\frac 12 \left(d(d+1)-(n+1)(d-1)\right)}&\text{for even $n$ and odd $d$}\\
		\frac 52 q^{\frac 12 \left(d(d+1)-(n+2)(d-2)\right)}&\text{for odd $n$ and even $d$}\\
		\frac 52 q^{\frac 12 \left(d(d+1)-(n+2)(d-1)\right)}&\text{for odd $n$ and odd $d$.}
	\end{cases}
	\]
	Then $R<1$, except for (i) $d=3$ and $n=5,6$, or (ii) $d=4$ and $n=9,10$, or (iii) $d=6$ and $n=9,10$. Corollary~\ref{cor:bounds_dualpolarspaces}~(e) and~\eqref{eqn:size_Steiner_d} imply that, in the respective cases, $R$ equals
	\[
	\frac{1+q^3}{1+q^4},\qquad \frac{(1+q^3)(1-q^8)}{1-q^{12}},\qquad \frac{(1-q^8)(1+q^5)}{1-q^{14}}.
	\]
	In all cases we have $R<1$, as required.
\end{proof}
\begin{proof}[Proof of (C5)]
	In this case we assume that $\cP=\hermitianpol{2n-1}$ with $2<d<n$. Use Corollary~\ref{cor:bounds_dualpolarspaces}~(a), \eqref{eqn:bound_Steiner_d}, and~\eqref{eq:bound_alpha} to obtain
	\[
	R<\begin{cases}
		\frac{14}{5} q^{(d-1)(d-n-1)}&\text{for odd $d$}\\
		\frac{14}{5} q^{(d-1)(d-n-1)+n}&\text{for even $d$}.
	\end{cases}
	\]
	Then $R<1$, except for $(n,d)=(5,4)$. In the latter case we find from Corollary~\ref{cor:bounds_dualpolarspaces}~(a), \eqref{eq:alpha_1}, \eqref{eqn:size_Steiner_d}, and~\eqref{eq:eps_bound} that
	\begin{align*}
		R&<\frac{q^8-1}{q^3-1} \prod_{i=1}^3 (q^{2i-1}+1)\frac{q^i+(-1)^i}{q^{5+i}+(-1)^i} \\
		&= \frac{(q^4-1)(q^5+1)}{(q^7+1)(q^3-1)}\leq 2q^{-1}\leq 1,
	\end{align*}
	as required.
\end{proof}
\begin{proof}[Proof of (C6)]
	In this case we assume that $\cP=\hermitianpol{2n}$ with $2<d<n-1$ or $d=n-1$ is odd, where the case $(n,d)=(6,4)$ is excluded. Use Corollary~\ref{cor:bounds_dualpolarspaces}~(b), \eqref{eqn:bound_Steiner_d}, and~\eqref{eq:bound_alpha} to obtain
	\begin{align}\label{eq:herm2_R}
		R<\begin{cases}
			\frac{14}{5} q^{(d-1)(d-n-2)+2d-1}&\text{for odd $d$}\\
			\frac{14}{5}q^{d(d-n-1)+2n+2}&\text{for even $d$}.
		\end{cases}
	\end{align}
	For odd $d$, it follows $R<1$, except when $(n,d)=(4,3)$. In the latter case we find from Corollary~\ref{cor:bounds_dualpolarspaces}~(b) and~\eqref{eqn:size_Steiner_d} that
	\[
	R=\frac{(q^4-1)(q^5+1)}{(q^3-1)(q^7+1)}<\frac{q^2+q^{-3}}{q^3-1}<1.
	\]
	\par
	If $d$ is even, then~\eqref{eq:herm2_R} implies $R<1$, except when $(n,d)=(8,6)$ (recall that we excluded $(n,d)=(6,4)$). In this case we find from Corollary~\ref{cor:bounds_dualpolarspaces}~(b),~\eqref{eq:alpha_1}, and~\eqref{eqn:size_Steiner_d} that
	\[
	R=\bigg(\prod_{i=1}^6 (1+q^{2i-1})\bigg)\bigg(\prod_{i=1}^4 \frac{q^i+(-1)^i}{q^{9+i}+(-1)^i}\bigg) \frac{q^5-1}{q^{14}-1}\,\epsilon(9,6)
	\]
	with
	\begin{align*}
		\epsilon(9,6)
		&=\frac{q^5+1+q\frac{q^{13}+1}{q^5-1}(q^4-1)}{q^5+1+q\frac{q^{13}+1}{q^{14}-1}(q^4-1)}\\
		&=\frac{q^{14}-1}{q^5-1}\;\frac{q^{18}-q^{14}+q^{10}+q^5-q-1}{q^{19}+q^{18}-q-1}\\
		&<\frac{q^{14}-1}{q^5-1}\;\frac 1q.
	\end{align*}
	This gives
	\begin{align*}
		R&<\frac 1q\bigg(\prod_{i=1}^6 (1+q^{2i-1})\bigg)\bigg(\prod_{i=1}^4 \frac{q^i+(-1)^i}{q^{9+i}+(-1)^i}\bigg)\\
		&=\frac 1q\,\frac{(q^8-1)(q^7+1)(q^9+1)}{(q^5-1)(q^6+1)(q^{13}-1)}\\
		&<q^{-3}\,\frac{q^7+1}{q^5-1}<1,
	\end{align*}
	using~\eqref{eq:ineqfraction}, which completes the proof.
\end{proof}
\par
Now it remains to prove the corner cases (C7)--(C9). In these cases we show that the dual distribution of the Steiner system has a negative entry, which contradicts~\eqref{eq:dualDistrnonnegative}. In what follows, all inner and dual distributions (in particular those in $\hermitianpol{2n-1}$) are determined with respect to the standard orderings imposed by~\eqref{eqn:def_Ri} and~\eqref{eq:Pnum_dualpolarspaces}. We require the following result on the inner and dual distributions of $t$-Steiner systems. 
\begin{proposition}\label{prop:tSteiner_tdesign_innerdistr}
	Let $X$ be the set of generators in a polar space of rank $n$, let $t$ be an integer satisfying $1\le t\le n$, and suppose that $Y$ is a $t$-Steiner system in $X$. Let $(A_i)$ and $(A'_k)$ be the inner distribution and dual distribution of $Y$, respectively. Then we have
	\begin{align*}
		A_{n-i}=\sum_{j=i}^{t-1} (-1)^{j-i} p^{\binom{j-i}{2}}\pbin{j}{i}\pbin{n}{j} \left(\prod_{\ell=j}^{t-1}(1+p^{n-\ell+e})-1\right)
	\end{align*}
	for all $i=0,1,\dots,n-1$ and $A'_1=A'_2=\cdots=A'_t=0$.	
\end{proposition}
\par
To prove Proposition~\ref{prop:tSteiner_tdesign_innerdistr} we use the following counterpart of Lemma~\ref{lem:QnumIdentity_Herm_hyp} for the $Q$-numbers of the association scheme of polar spaces, for which we give a proof in the appendix.
\begin{lemma}\label{lem:identityQnum_polar}
	Let $X$ be the set of generators in a polar space of rank $n$ and let $Q_k(i)$ be the corresponding $Q$-numbers given by~\eqref{eq:Pnum_dualpolarspaces} and~\eqref{eq:AskeyWilsonDuality}. Then we have
	\[
	\sum_{k=0}^n p^{k(n-j)}\pbin{n-k}{n-j}\prod_{\ell=1}^{n-j} (1+p^{\ell-k+e}) Q_k(i)=|X|\pbin{n-i}{j}
	\]
	for all $i,j=0,1,\dots,n$.
\end{lemma}
\par
We now prove Proposition~\ref{prop:tSteiner_tdesign_innerdistr}.
\begin{proof}[Proof of Proposition~\ref{prop:tSteiner_tdesign_innerdistr}]
	From~\eqref{eq:dualdistr} and Lemma~\ref{lem:identityQnum_polar} we find that, for all $j\ge 0$,
	\begin{align}\label{eq:identity_dualinner_Steiner}
		\sum_{k=0}^j A'_k\,p^{k(n-j)}\pbin{n-k}{n-j} \prod_{\ell=1}^{n-j}(1+p^{\ell-k+e})=|X|\sum_{i=0}^n A_i \pbin{n-i}{j}.
	\end{align}
	Since $Y$ is an $(n-t+1)$-code, we have $A_0=1$ and $A_1=\cdots=A_{n-t}=0$ and therefore obtain, by setting $j=t$ in~\eqref{eq:identity_dualinner_Steiner},
	\[
	\sum_{k=0}^{t}A'_k\, p^{k(n-t)}\pbin{n-k}{n-t} \prod_{\ell=1}^{n-t}(1+p^{\ell-k+e})=|X|\pbin{n}{t}.
	\]
	From $A'_0=|Y|$ we find that
	\[
	\sum_{k=1}^{t}A'_k\, p^{k(n-t)}\pbin{n-k}{n-t} \prod_{i=1}^{n-t}(1+p^{\ell-k+e})=\pbin{n}{t}\left(|X|-|Y|\prod_{\ell=1}^{n-t}(1+p^{\ell+e})\right).
	\]
	From the expression~\eqref{eqn:size_X} for $|X|$ and the expression~\eqref{eqn:size_Steiner_t} for $|Y|$, we see that the right-hand side is zero. Since $A'_k\ge 0$ by~\eqref{eq:dualDistrnonnegative}, we conclude $A'_1=A'_2=\cdots=A'_t=0$. Moreover~\eqref{eq:identity_dualinner_Steiner} simplifies to
	\[
	\pbin{n}{j} \prod_{\ell=1}^{n-j}(1+p^{\ell+e})|Y|=|X|\left(\pbin{n}{j}+\sum_{i=n-t+1}^n A_i \pbin{n-i}{j}\right)
	\]
	for $j=0,1,\dots,t-1$. Using~\eqref{eqn:size_X} and the expression~\eqref{eqn:size_Steiner_t} for $|Y|$ again, we obtain
	\begin{align*}
		\sum_{i=0}^{t-1} A_{n-i} \pbin{i}{j}
		&=\pbin{n}{j}\left(\prod_{\ell=j}^{t-1}(1+p^{n-\ell+e})-1\right).
	\end{align*}
	By the inversion formula
	\begin{align}\label{eq:qbininversion}
		\sum_{j=i}^k  (-1)^{j-i}q^{\binom{j-i}{2}}\qbin{j}{i}\qbin{k}{j}=\delta_{ik}
	\end{align}
	for nonnegative integers $i,k$ (which can be deduced from the $q$-binomial theorem, for example), we obtain the desired expression for $A_{n-i}$.
\end{proof}
\par
We now prove (C7)--(C9). Henceforth we denote by $(A_i)$ and $(A'_k)$ the inner and dual distribution, respectively, of a putative $t$-Steiner system $Y$.
\begin{proof}[Proof of (C7)]
	We now assume that $t=2$ and $\cP=B_n$ or $C_n$ for odd $n>3$ or $\cP=\elliptic{n+1}$ for even $n>3$. We will show that $A'_{n-1}<0$ in the first case and $A'_n<0$ in the second case. By~\eqref{eq:dualdistr} and~\eqref{eq:AskeyWilsonDuality} we have
	\[
	\frac{A'_k}{\mu_k}=1+\frac{P_{n-1}(k)}{v_{n-1}}A_{n-1}+\frac{P_n(k)}{v_n}A_n.
	\]
	By Proposition~\ref{prop:tSteiner_tdesign_innerdistr} we have
	\begin{align*}
		A_{n-1}=q^{n-1+e}\qbin{n}{1}\qandq
		A_n=(q^{n+e}+1)(q^{n-1+e}+1)-\qbin{n}{1}q^{n-1+e}-1.
	\end{align*}
	From~\eqref{eq:Pnum_dualpolarspaces} and~\eqref{eqn:valuencies} we find for $B_n$ and $C_n$ that
	\[
	\frac{P_{n-1}(n-1)}{v_{n-1}}=\qbin{n}{1}^{-1}\left(q^{-n+1}-q^{-2n+4}\qbin{n-1}{1}\right)\quad\text{and}\quad \frac{P_n(n-1)}{v_n}=q^{-2n+2},
	\]
	and for $\elliptic{n+1}$ that
	\begin{align*}
		\frac{P_{n-1}(n)}{v_{n-1}}=-q^{-2n+2}\qandq
		\frac{P_n(n)}{v_n}=q^{-2n}.
	\end{align*}
	Here we have crucially used the assumed parity of $n$. For $B_n$ and $C_n$, we then obtain
	\[
	\frac{A'_{n-1}}{\mu_{n-1}}
	=2-\frac{q^n-1}{(q-1)q^{n-1}}-\frac{1}{q^{2n-2}}-\frac{q^{n-1}-1}{(q-1)q^{n-3}}+\frac{(q^n+1)(q^{n-1}+1)}{q^{2n-2}}.
	\]
	For $n>3$, we have
	\begin{align*}
		2-\frac{q^n-1}{(q-1)q^{n-1}}-\frac{1}{q^{2n-2}}
		&=\frac{q^{2n-1}-2q^{2n-2}+q^{n-1}-q+1}{(q-1)q^{2n-2}}\\
		&<\frac{q^{2n-1}-2q^{n+1}+q^{n-1}-q+1}{(q-1)q^{2n-2}}\\
		&=\frac{q^{n-1}-1}{(q-1)q^{n-3}}-\frac{(q^n+1)(q^{n-1}+1)}{q^{2n-2}}
	\end{align*}
	and therefore $A'_{n-1}<0$ if $\cP=B_n$ or $C_n$, which completes the proof in the first case. For $\elliptic{n+1}$, we obtain
	\[
	\frac{A'_n}{\mu_n}=1-\frac{q^n-1}{(q-1)q^n}-\frac{1}{q^{2n}}-\frac{(q^n-1)q^2}{(q-1)q^n}+\frac{(1+q^{n+1})(1+q^n)}{q^{2n}}.
	\]
	For $n>2$, we have
	\begin{align*}
		1-\frac{q^n-1}{(q-1)q^n}-\frac{1}{q^{2n}}
		&=\frac{q^{2n+1}-2q^{2n}+q^n-q+1}{(q-1)q^{2n}}\\
		&<\frac{q^{2n+1}-2q^{n+2}+q^n-q+1}{(q-1)q^{2n}}\\
		&=\frac{(q^n-1)q^2}{(q-1)q^n}-\frac{(1+q^{n+1})(1+q^n)}{q^{2n}},
	\end{align*}
	and therefore $A'_n<0$ in the case $\cP=\elliptic{n+1}$. This completes the proof.
\end{proof}
\begin{proof}[Proof of (C8)]
	We now assume $\cP=\elliptic{n+1}$ for $t=3$ and even $n>4$. As in (C7) we compute
	\[
	\frac{A'_{n-1}(q-1)^2(q+1)}{2\mu_{n-1}}=-q(q+1)(1-q^{2-n})(1-q^{4-n})+q^{5-3n}(1+q^{-2}),
	\]
	from which it is readily verified that $A'_{n-1}<0$, as required.
\end{proof} 	
\begin{proof}[Proof of (C9)]
	As in (C7) we compute the following. For $\cP=\elliptic{8}$ and $t=4$, we have
	\begin{align*}
		\frac{A'_6}{\mu_6}=-2q^{-5} (q + 1)^2 (q^2 + 1)(q^3 + q + 1)<0,
	\end{align*}
	for $\elliptic{9}$ and $t=5$, we have
	\begin{align*}
		\frac{A'_7}{\mu_7}=-2q^{-5} (q + 1)^4 (q^2 - q + 1) (q^2 + 1)^2<0,
	\end{align*}
	and for $\hermitianpol{12}$ and $t=3$, we have
	\begin{align*}
		\frac{A'_5}{\mu_5}=-q^{-7}(q + 1)^3 (q^2 - q + 1) (q^4 - q^3 + q^2 + 1)<0.
	\end{align*}
	In all cases we obtain the required nonexistence of $t$-Steiner systems.
\end{proof}

\appendix\section{Appendix: Identities for the \texorpdfstring{$Q$}{Q}-numbers}\label{appendix}

We now prove Lemmas~\ref {lem:QnumIdentity_Herm_hyp} and~\ref{lem:identityQnum_polar}. We will frequently use the identity
\[
\qbin{k}{j}\qbin{j}{i}=\qbin{k}{i}\qbin{k-i}{j-i}
\]
without specific reference.
\begin{proof}[Proof of Lemma~\ref{lem:identityQnum_polar}]
	Let $P_i(k)$ be as in~\eqref{eq:Pnum_dualpolarspaces} and $Q_k(i)$ be the corresponding $Q$-number, determined by~\eqref{eq:AskeyWilsonDuality}. We will prove
	\begin{align}\label{eq:identity_Pnum_polar}
		\sum_{i=0}^n \pbin{n-i}{j} P_i(k)
		=p^{k(n-j)}\pbin{n-k}{n-j}\prod_{\ell=1}^{n-j} (1+p^{\ell-k+e}).
	\end{align}
	By multiplying~\eqref{eq:identity_Pnum_polar} with $Q_k(\ell)$, taking the sum over~$k$, and using~\eqref{eq:PQnumorth}, we obtain the identity in the lemma. It remains to prove~\eqref{eq:identity_Pnum_polar}. For all $i,j=0,1,\dots,n$, we have
	\begin{align*}
		\sum_{i=0}^n \pbin{n-i}{j} P_i(k)
		&=\sum_{i=0}^n\sum_{\ell=0}^i (-1)^\ell\pbin{n}{k}^{-1}\pbin{n-i}{j}\pbin{n-i}{k-\ell}\pbin{n}{i}\pbin{i}{\ell} p^{\ell(\ell-i-e-1)+\binom{i+1}{2}+ie}\\
		&=\sum_{i=0}^n\sum_{\ell=0}^i (-1)^\ell\pbin{n}{k}^{-1}\pbin{n-i}{j}\pbin{n}{\ell}\pbin{n-\ell}{k-\ell}\pbin{n-k}{i-\ell} p^{\ell(\ell-i-e-1)+\binom{i+1}{2}+ie}.
	\end{align*}
	Interchanging the order of summation by putting $m=i-\ell$ gives us
	\begin{equation}\label{eqn:sum_P}
		\sum_{i=0}^n \pbin{n-i}{j} P_i(k)
		=\sum_{m=0}^{n-k}\left(\sum_{\ell=0}^k (-1)^\ell p^{\binom{\ell}{2}} \pbin{k}{\ell}\pbin{n-m-\ell}{j}\right)\pbin{n-k}{m} p^{\binom{m}{2}+m(e+1)}.
	\end{equation}
	To evaluate the inner sum, we use the $q$-Chu-Vandermonde identity
	\begin{align}\label{eq:qChu-Vandermonde1}
		\pbin{x+y}{z}=\sum_{i=0}^x p^{i(y-z+i)}\pbin{x}{i}\pbin{y}{z-i},
	\end{align}
	where $x,y,z$ are integers. Applying the inversion formula~\eqref{eq:qbininversion} to~\eqref{eq:qChu-Vandermonde1} reveals that
	\[
	\sum_{\ell=0}^x (-1)^\ell p^{\binom{\ell}{2}} \pbin{x}{\ell}\pbin{x-\ell+y}{z}=p^{x(y-z+x)}\pbin{y}{z-x}.
	\]
	Put $x=k$, $y=n-k-m$ and $z=j$ to obtain
	\[
	\sum_{\ell=0}^k (-1)^\ell p^{\binom{\ell}{2}}\pbin{k}{\ell}\pbin{n-m-\ell}{j}=p^{k(n-m-j)}\pbin{n-k-m}{j-k}.
	\]
	Substitute into~\eqref{eqn:sum_P} to give
	\begin{align}
		\sum_{i=0}^n \pbin{n-i}{j} P_i(k)
		&=\sum_{m=0}^{n-k}p^{k(n-m-j)}\pbin{n-k-m}{j-k}\pbin{n-k}{m} p^{\binom{m}{2}+m(e+1)}\notag\\
		&=\left(\sum_{m=0}^{n-j}\pbin{n-j}{m}p^{\binom{m}{2}+m(e+1)-km}\right) p^{k(n-j)}\pbin{n-k}{j-k}.\label{eq:Pnum_identity_proof1}
	\end{align}
	Applying the $q$-binomial theorem
	\[
	\sum_{i=0}^k q^{\binom{i}{2}}\qbin{k}{i}z^i=\prod_{i=0}^{k-1} (1+zq^i)
	\]
	to the sum on the right-hand side of~\eqref{eq:Pnum_identity_proof1} leads to the identity~\eqref{eq:identity_Pnum_polar}.
\end{proof}
\par
To prove Lemma~\ref{lem:QnumIdentity_Herm_hyp}, we require some identities involving the $q$-Pochhammer symbol. For a real number~$a$ and nonnegative integers $n,k$, we have
\begin{align}
	\qbin{n}{k}=&\,\frac{\pochq{q^{-n}}{k}}{\pochq{q}{k}} (-1)^k q^{kn-\binom{k}{2}}\label{eq:qbin_poch}\\
	\poch{a^2}{q^2}{k}=&\,\pochq{a}{k}\pochq{-a}{k}\label{eq:pochidentity_squared}\\
	\pochq{a}{2k}=&\,\poch{a}{q^2}{k}\poch{aq}{q^2}{k}\label{eq:pochidentity_2k}\\
	\pochq{a}{n+k}=&\,\pochq{a}{n}\pochq{aq^n}{k}\label{eq:pochidentity_indexsum}\\
	\pochq{a}{n-k}=\frac{\pochq{a}{n}}{\pochq{a^{-1}q^{1-n}}{k}}& (-a)^{-k} q^{\binom{k}{2}-nk+k}\quad\text{for $a\neq 0$} \label{eq:pochidentity_indexdiff}.
\end{align}
These identities can be found in~\cite{Koekoek}, for example.
\begin{proof}[Proof of Lemma~\ref{lem:QnumIdentity_Herm_hyp}]
	Let $P'_i(k)$  and $Q'_k(i)$ be as in~\eqref{eq:Pnum_Herm_hyp} and~\eqref{eq:Qnum_Herm_hyp}, respectively. To simplify notation, we set $a=q^{-1}c^{-1}b^{-2n}$. Recall that $(x)_i=(x;b)_i$. We will show that
	\begin{align}\label{eq:Pnumidentity_Herm_hyp}
		\sum_{i=0}^n \bbin{n-i}{j}P'_i(k)=b^{k(n-j)}\qbin{n-k}{n-j} \frac{(a^{-1} b^{-n-k})_{n-j}}{(q)_{n-j}}.
	\end{align}
	Multiplying~\eqref{eq:Pnumidentity_Herm_hyp} with $Q'_k(\ell)$, taking the sum over $k$, and using~\eqref{eq:PQnumorth} we obtain the identity in the lemma. It remains to prove~\eqref{eq:Pnumidentity_Herm_hyp}. First we rewrite the valencies~$v'_i$, given in Table~\ref{table:Pnum_parameters}, such that we have a similar form for $P'_i(k)$ in both association schemes. For $\hermitianpol{2n-1}$, we use \eqref{eq:qbin_poch} and \eqref{eq:pochidentity_squared} to obtain
	\begin{align*}
		v'_i=q^{i^2}\qqbin{n}{i}
		&=(-1)^iq^{i^2-2\binom{i}{2}+2ni} \frac{\poch{(-q)^{-n}}{-q}{i}\poch{-(-q)^{-n}}{-q}{i}}{\poch{-q}{-q}{i}\poch{q}{-q}{i}}\\
		&=(-1)^i (-q)^{\binom{i}{2}+i+ni}\gaussian{-q}{n}{i} \frac{\poch{-(-q)^{-n}}{-q}{i}}{\poch{q}{-q}{i}}.
	\end{align*}
	For $\hyperbolichalf$, we use~\eqref{eq:qbin_poch} and~\eqref{eq:pochidentity_2k} to obtain
	\begin{align*}
		v'_i&=q^{\binom{2i}{2}}\qbin{m}{2i}
		=q^{2im} \frac{\poch{q^{-m}}{q^2}{i}\poch{q^{-m+1}}{q^2}{i}}{\poch{q^2}{q^2}{i}\poch{q}{q^2}{i}}.
	\end{align*}
	For even $m=2n$, we have
	\[
	v'_i=(-1)^iq^{2in+2\binom{i}{2}}\qqbin{n}{i} \frac{\poch{q^{-2n+1}}{q^2}{i}}{\poch{q}{q^2}{i}}
	\]
	and for odd $m=2n+1$, we obtain
	\[
	v'_i=(-1)^iq^{2in+2i+2\binom{i}{2}}\qqbin{n}{i} \frac{\poch{q^{-2n-1}}{q^2}{i}}{\poch{q}{q^2}{i}}.
	\]
	Hence, in all cases, we can write
	\[
	v'_i=(-q)^i c^i b^{\binom{i}{2}+ni} \bbin{n}{i} \frac{\pochb{c^{-1}b^{-n}}{i}}{\pochb{q}{i}}.
	\]
	Now, from the expression~\eqref{eq:Pnum_Herm_hyp} for $P'_i(k)$, we obtain
	\begin{align*}
		\sum_{i=0}^n \bbin{n-i}{j}P'_i(k)
		&=\sum_{i=0}^n \bbin{n-i}{j}(-q)^i c^i b^{\binom{i}{2}+ni}  \bbin{n}{i} \frac{(c^{-1}b^{-n})_i}{(q)_i} \hyp{b^{-i},b^{-k},a b^k}{b^{-n},c^{-1}b^{-n}}{b;b}\\
		&=\bbin{n}{j}\sum_{i,\ell\geq 0} \bbin{n-j}{i}(-q)^i c^i b^{\binom{i}{2}+ni+\ell} \frac{(c^{-1}b^{-n})_i (b^{-i})_\ell (b^{-k})_\ell (a b^k)_\ell}{(q)_i (b^{-n})_\ell (c^{-1}b^{-n})_\ell (b)_\ell}.
	\end{align*}
	From~\eqref{eq:qbin_poch} we have
	\[
	\bbin{n-j}{i}\frac{(b^{-i})_\ell}{(b)_\ell}
	=(-1)^\ell b^{\binom{\ell}{2}-i\ell}\bbin{n-j}{\ell}\bbin{n-j-\ell}{i-\ell},
	\]
	and therefore
	\begin{align}\label{eq:identityPnum1}
		\sum_{i=0}^n \bbin{n-i}{j}P'_i(k)=
		\bbin{n}{j}\sum_{\ell\geq 0} (-1)^\ell b^{\binom{\ell}{2}+\ell}\bbin{n-j}{\ell} \frac{(b^{-k})_\ell (a b^k)_\ell}{(b^{-n})_\ell (c^{-1}b^{-n})_\ell}\; S_\ell,
	\end{align}
	where
	\[
	S_\ell=\sum_{i\geq0} (-q)^i c^i b^{\binom{i}{2}+i(n-\ell)} \bbin{n-j-\ell}{i-\ell} \frac{(c^{-1}b^{-n})_i }{(q)_i} .
	\]
	By interchanging the order of summation and then applying~\eqref{eq:pochidentity_indexsum}, we obtain
	\begin{align*}
		S_\ell&=\sum_{i=0}^{n-\ell}(-q)^{i+\ell} c^{i+\ell} b^{\binom{i+\ell}{2}+(i+\ell)(n-\ell)}\bbin{n-j-\ell}{i} \frac{(c^{-1}b^{-n})_{i+\ell}}{(q)_{i+\ell}}\\
		&=\sum_{i=0}^{n-\ell} (-q)^{i+\ell} c^{i+\ell} b^{\binom{i+\ell}{2}+(i+\ell)(n-\ell)}\bbin{n-j-\ell}{i} \frac{(c^{-1}b^{-n})_\ell(c^{-1}b^{-n+\ell})_i}{(q)_\ell(q b^\ell)_i}.
	\end{align*}
	Using~\eqref{eq:qbin_poch}, this sum becomes
	\begin{align*}
		S_\ell&=(-q)^\ell c^\ell b^{\binom{\ell}{2}-\ell^2+n\ell}\frac{(c^{-1}b^{-n})_\ell}{(q)_\ell}
		\sum_{i=0}^{n-\ell} (qcb^{2n-j-\ell})^i \frac{(b^{-(n-j-\ell)})_i(c^{-1} b^{-n+\ell})_i}{(b)_i(qb^\ell)_i}\\
		&=(-q)^\ell c^\ell b^{\binom{\ell}{2}-\ell^2+n\ell}\frac{(c^{-1}b^{-n})_\ell}{(q)_\ell} 
		\hyper{2}{1}{b^{-(n-j-\ell)},c^{-1}b^{-n+\ell}}{qb^\ell}{b;qcb^{2n-j-\ell}}.
	\end{align*}
	The hypergeometric function $\mbox{}_2\phi_1$ can be evaluated by using the $q$--Chu--Vandermonde identity
	\[
	\hyper{2}{1}{b^{-k},x}{y}{b;\frac{yb^k}{x}}=\frac{(x^{-1}y)_k}{(y)_k}
	\]
	(see \cite[(1.11.4)]{Koekoek}, for example), which implies that
	\[
	S_\ell=(-q)^\ell c^\ell b^{\binom{\ell}{2}-\ell^2+n\ell}\frac{(c^{-1}b^{-n})_\ell (qcb^n)_{n-j-\ell}}{(q)_\ell (qb^\ell)_{n-j-\ell}}.
	\]
	Substitute into~\eqref{eq:identityPnum1} to obtain
	\begin{align}\label{eq:sumwoname}
		\sum_{i=0}^n \bbin{n-i}{j}P'_i(k)
		=\bbin{n}{j}\sum_{\ell\geq 0} q^\ell c^\ell b^{n\ell}\bbin{n-j}{\ell}   \frac{(b^{-k})_\ell (a b^k)_\ell (qcb^n)_{n-j-\ell}}{(b^{-n})_\ell (q)_\ell (qb^\ell)_{n-j-\ell}}.
	\end{align}
	From~\eqref{eq:pochidentity_indexsum} we have
	\begin{align}\label{eq:poch1}
		(qb^\ell)_{n-j-\ell}=\frac{(q)_{n-j}}{(q)_\ell}
	\end{align}
	and from~\eqref{eq:pochidentity_indexdiff} we find that
	\begin{align}\label{eq:poch2}
		(qcb^n)_{n-j-\ell}=\frac{(qcb^n)_{n-j}}{(q^{-1}c^{-1}b^{-2n+1+j})_\ell} (-qcb^n)^{-\ell} b^{\binom{\ell}{2}-(n-j)\ell+\ell}.
	\end{align}
	By substituting~\eqref{eq:poch1} and~\eqref{eq:poch2} into~\eqref{eq:sumwoname} and using~\eqref{eq:qbin_poch}, we have
	\begin{align*}
		\sum_{i=0}^n \bbin{n-i}{j}P'_i(k)
		&=\bbin{n}{j}\sum_{\ell\geq 0}(-1)^{\ell} b^{\binom{\ell}{2}-(n-j)\ell+\ell}  \bbin{n-j}{\ell} \frac{(b^{-k})_\ell (a b^k)_\ell (qcb^n)_{n-j}}{(b^{-n})_\ell (q)_{n-j}(q^{-1}c^{-1}b^{-2n+1+j})_\ell}\\
		&=\bbin{n}{j}\frac{(qcb^n)_{n-j}}{(q)_{n-j}}\sum_{\ell\geq 0}  b^\ell  \frac{(b^{-(n-j)})_\ell (b^{-k})_\ell (a b^k)_\ell }{(b)_\ell (b^{-n})_\ell (q^{-1}c^{-1} b^{-2n+1+j})_\ell}\\
		&=\bbin{n}{j}\frac{(qcb^n)_{n-j}}{(q)_{n-j}}\hyp{b^{-(n-j)},b^{-k},a b^k}{b^{-n},q^{-1}c^{-1} b^{-2n+1+j}}{b;b}.
	\end{align*}
	The hypergeometric function $\mbox{}_3\phi_2$ on the right hand side can be computed via the $q$--Pfaff--Saalsch\"utz formula
	\[
	\hyp{b^{-i},x,y}{z,xyz^{-1}b^{1-i}}{b;b}=\frac{(x^{-1}z)_i (y^{-1}z)_i}{(z)_i (x^{-1}y^{-1}z)_i}
	\]
	(see \cite[(1.11.9)]{Koekoek}, for example). Note that $qcb^n=a^{-1}b^{-n}$. Therefore, we obtain
	\[
	\sum_{i=0}^n \bbin{n-i}{j}P'_i(k)=\bbin{n}{j}\frac{(b^{-(n-k)})_{n-j} (qcb^{n-k})_{n-j}}{(q)_{n-j}(b^{-n})_{n-j}}.
	\]
	Applying~\eqref{eq:qbin_poch} to $\bbin{n}{j}=\bbin{n}{n-j}$ and using~\eqref{eq:qbin_poch} one more time leads to the identity~\eqref{eq:Pnumidentity_Herm_hyp}.
\end{proof}

%

\newcommand{\etalchar}[1]{$^{#1}$}
\providecommand{\bysame}{\leavevmode\hbox to3em{\hrulefill}\thinspace}
\providecommand{\MR}{\relax\ifhmode\unskip\space\fi MR }
\providecommand{\MRhref}[2]{%
	\href{http://www.ams.org/mathscinet-getitem?mr=#1}{#2}
}
\providecommand{\href}[2]{#2}

\end{document}